\pdfoutput=1
\documentclass{amsart}
\usepackage{amscd}
\usepackage{tikz}
\usepackage{tikz-cd}
\usetikzlibrary{matrix,arrows,shapes,decorations.pathmorphing}

\newtheorem{thm}{Theorem}[section]

\newtheorem{prop}[thm]{Proposition}
\newtheorem{conj}[thm]{Conjecture}
\theoremstyle{definition}
\newtheorem{defn}[thm]{Definition}
\theoremstyle{remark}

\newcommand{\example}[1]{\mathbf{A}_{#1}} 
\newcommand{\examplep}[1]{\mathbf{Q}_{#1}} 

\newcommand{\shpen}[2]{U_{{#1},{#2}}} 
\newcommand{\qshpen}[2]{\mcu_{{#1},{#2}}}
\newcommand{\fs}[1]{\mathcal{F}^\rightharpoonup (#1 )}

\newcommand{\ptgroup}[1]{\mathbb{L}_{#1}} 
\newcommand{\tgroup}[1]{\mathbb{G}_{#1}} 

\newcommand{\linsys}[1]{{\mcl_{#1}}} 
\newcommand{\linsysfull}[1]{{\mcl_{#1}^f}} 
\newcommand{\linsysveryfull}[1]{{\mcl_{#1}^{vf}}} 

\newcommand{\du}[1]{{\overline{#1}}} 

\newcommand{\psec}[1]{{\Sigma (#1 )}}
\newcommand{\ptlaf}[1]{{\Theta (#1 )}} 
\newcommand{\fans}[1]{{\mcf_{\Sigma (#1 )}}} 
\newcommand{\fanlt}[1]{{\mcf_{\ptlaf{#1}}}} 

\newcommand{\mlg}[2]{{\mathcal{M}_{#1 , #2}}} 

\newcommand{\laf}[1]{{\mathcal{X}}_{\ptlaf{ #1 }}} 
\newcommand{\secon}[1]{\mathcal{X}_{\Sigma (#1 )}} 

\newcommand{\hyp}[1]{\mathcal{Y}_{#1}} 
\newcommand{\fib}[2]{\mathcal{Z}_{#1} (#2 )} 

\newcommand{\orb}[1]{{\text{orb} (#1 )}}

\newcommand{\convhull}{{\text{Conv}}}

\newcommand{\ext}{{\text{Ext}}}

\newcommand{\Hom}{{\text{Hom}}}

\newcommand{\tor}{{\text{Tor}}}

\newcommand{\cone}{{\text{Cone}}}

\newcommand{\Span}{{\text{Span}}}

\newcommand{\R}{\mathbb{R}}
\newcommand{\C}{\mathbb{C}}

\newcommand{\p}{\mathbb{P}}

\newcommand{\Z}{\mathbb{Z}}

\newcommand{\mcb}{\mathcal{B}}
\newcommand{\mcc}{\mathcal{C}}
\newcommand{\mcd}{\mathcal{D}}
\newcommand{\mce}{\mathcal{E}}
\newcommand{\mcf}{\mathcal{F}}

\newcommand{\mch}{\mathcal{H}}
\newcommand{\mci}{\mathcal{I}}

\newcommand{\mcl}{\mathcal{L}}
\newcommand{\mco}{\mathcal{O}}
\newcommand{\mcp}{\mathcal{P}}
\newcommand{\mcq}{\mathcal{Q}}
\newcommand{\mcm}{\mathcal{M}}

\newcommand{\mcr}{\mathcal{R}}
\newcommand{\mcs}{\mathcal{S}}

\newcommand{\mcu}{\mathcal{U}}
\newcommand{\mcv}{\mathcal{V}}

\newcommand{\mcx}{\mathcal{X}}
\newcommand{\mcy}{\mathcal{Y}}
\newcommand{\mcz}{\mathcal{Z}}

\begin{document}

\title{Compactifications of spaces of Landau-Ginzburg models}


\author[C. Diemer]{Colin Diemer}
\address{Department of Mathematics, University of Miami, Coral Gables, FL, 33146, USA}
\email{diemer@math.miami.edu}

\author[L. Katzarkov]{Ludmil Katzarkov}
\address{Fakult\"at f\"ur Mathematik , Universit\"at Wien, 1090 Wien, Austria}
\email{ludmil.katzarkov@univie.ac.at}


\author[G. Kerr]{Gabriel Kerr}
\address{Department of Mathematics, University of Miami, Coral Gables, FL, 33146, USA}
\email{gdkerr@math.miami.edu}





\begin{abstract}
This paper reviews results and techniques from \cite{DKK} and applies them in basic examples of Landau-Ginzburg models. The main example is the $A_n$ category where we observe a relationship to stability conditions and directed quiver representations. We conclude with a brief survey of applications to the birational geometry of del Pezzo surfaces.
\end{abstract}

\maketitle
\section{Introduction}

One case of homological mirror symmetry is an equivalence between the derived category of coherent sheaves on a Fano variety $X$ and the Fukaya-Seidel category of its mirror Landau-Ginzburg, or LG, model $\mathbf{w} : X^{mir} \to \C$. There are many constructions of the mirror \cite{auroux}, \cite{HV} but all depend on a choice of symplectic form on $X$. Moving within the complexified K\"ahler cone of $X$ gives an open parameter space of mirror LG models. While the Fukaya-Seidel categories of any two  mirror LG models from this space are equivalent, we may assign distinct exceptional collections and semi-orthogonal decompositions to certain regions. We observe that these decompositions should be related to the space of stability conditions of $D^b (X)$. For more on LG models from this vantage point, see \cite{HKK}, \cite{IKS} , \cite{KKSP} and \cite{CP}.

In \cite{DKK}, \cite{Kerr}, the authors examine this phenomena in the toric context and compactify the space of LG models into a toric stack $\mlg{A}{A^\prime}$ using methods from \cite{GKZ}, \cite{Lafforgue}. The boundary of this stack gives degenerations of the LG models where both the fiber and the base of the model degenerate. Examining the fixed points of $\mlg{A}{A^\prime}$, we see a LG model decompose into a chain of regenerated circuit LG models. In \cite{DKK} we considered the symplectic topology of these degenerated pieces and observed that, under homological mirror symmetry, they correspond to semi-orthogonal components of $D^b(X)$ obtained by running the toric minimal model program on the mirror toric Fano $X$. The mirror symmetric decomposition is a run of the minimal model program on $X$. We expect that this type of correspondence between Mori theoretic semi-orthogonal decompositions and degenerated Landau-Ginzburg models holds in much more generality, leading to a new approach to birational geometry.

In addition to reviewing the definitions and theorems in the approach outlined above, we give a detailed account of a basic example, which is still rich in structure.  Here our LG models are simply single variable degree $(n + 1)$ polynomials $f:\mathbb{C}\rightarrow\mathbb{C}$, and the Newton polytope is simply the interval $[0,n + 1]$. Because the fibers are finite sets, this allows us to set aside much of the technical symplectic topology in \cite{DKK}. The secondary polytope of the interval was investigated in \cite{GKZ}, where it was shown to be a cube whose lattice structure is strongly related to $A_n$ representation theory. We investigate the universal family over the toric stack of this cube which was identified with a quotient of the Losev-Manin space in  \cite{blume}.  Finally, we analyze the monotone path polytope in this setting, as well as the combinatorial structure of the vanishing thimbles near the degenerated LG models. After this careful study, we observe connections with the classical representations of $A_n$ quivers and an interpretation of reflection functors as  wall crossing in the space of stability conditions of the $A_n$ category.

In the final section of the paper, we explore the homological mirror of the three point blow up $X$ of $\p^2$. We build on the work of \cite{kalo} which studied relations between Sarkisov links. In the usual setup of Sarkisov links, earlier contractions do not play a prominent role. We explain how the toric compactification of the LG model mirror of $X$ preserves this data and gives a more complete picture of the minimal model program for $X$.

\vspace{2mm}{\em Acknowledgements:}  The authors would like to thank D. Auroux, M. Ballard, C. Doran, D. Favero, M. Gross, F. Haiden, A. Iliev, S. Keel, M. Kontsevich, J. Lewis, T. Pantev, C. Prizhalkovskii, H. Rudatt, E. Scheidegger, Y. Soibelman and G. Tian for valuable comments and suggestions.

\section{\label{sec:toric} Toric Landau-Ginzburg models}

In this section we review constructions from \cite{DKK} which compactify the moduli of hypersurfaces in a toric stack and a moduli space of LG models. This is followed by a detailed definition of radar screens, which are distinguished bases for the LG models designed to preserve categories in the degenerated models. The choices involved in defining these bases are condensed into a torsor over the monotone path stack.

\subsection{\label{sec:tslg} Toric stacks and LG models}
We start by introducing the toric machinery that we need for the rest of the paper. Letting $M$ be a rank $d$ lattice and $A$ a finite subset in $ M$, we take $\du{A} \subset N$ to be the collection of primitives normal vectors to facets of $Q = \convhull ( A)$. Here we use the usual notation of $N = \Hom (M , \Z)$ and write $\convhull (A)$ for the convex hull of a set of points. The normal fan $\mcf_Q$ of $Q$ has $\du{A}$ as the set of generators for one cones and defines an abstract simplicial complex structure on the set $\du{A}$. We take the toric variety $X_Q$ to be the variety associated to $\mcf_Q$. 

To promote this to a toric stack, we follow the prescription given in \cite{BCS} and \cite{Cox}. Define $U_Q \subset \C^{\du{A}}$ to be the open toric variety given by taking the fan in $\R^{\du{A}}$ consisting of cones $ \cone \{e_\alpha : \alpha \in \sigma\}$ where $\sigma$ is any cone in the normal fan $\mcf_Q$ of $Q$.  The map $\beta_{\du{A}} : \Z^{\du{A}} \to N$ is defined to take $e_\alpha$ to $\alpha$ and we write its kernel  and cokernel as $L_{\du{A}}$ and $K_{\du{A}}$. Define the group $\ptgroup{Q} = (L_{\du{A}} \otimes \C^*) \oplus \tor (K_{\du{A}}, \C^* )$ as a subgroup of $(\C^*)^{\du{A}}$ using the inclusion and connecting homomorphism  to obtain the stack
\begin{equation*} 
\mcx_Q = [U_Q / \ptgroup{Q}] .
\end{equation*}
The variety $X_Q$ is the coarse space of $\mcx_Q$. As in the case of toric varieties defined from polytopes, the stack $\mcx_Q$ comes equipped with a line bundle $\mco_Q (1)$. Letting $Q_\Z = Q \cap M$, the space of sections $H^0 (\mcx_Q , \mco_Q (1))$ has an equivariant basis $\{s_\alpha :\alpha \in Q_\Z \}$ and a linear system $\linsys{A} = \Span \{s_\alpha : \alpha \in A\}$. We distinguish two open subsets of $\linsys{A}$, the full sections
\begin{equation*} \linsysfull{A} =  \{s = \sum c_\alpha s_\alpha : c_\alpha \ne 0 , \text{ for all vertices } \alpha \in Q\}, \end{equation*}
and the very full sections 
\begin{equation*} \linsysveryfull{A} = \{s = \sum c_\alpha s_\alpha : c_\alpha \ne 0 \text{ for every } \alpha \in A\}. \end{equation*}

\begin{figure}
\begin{picture}(0,0)%
\includegraphics{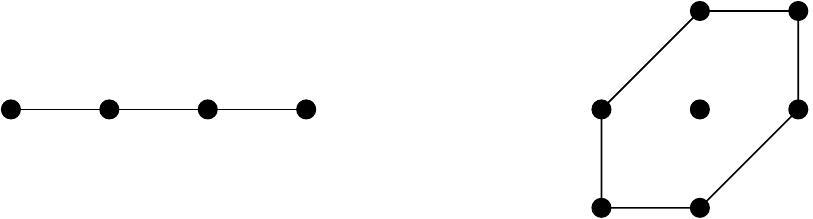}%
\end{picture}%
\setlength{\unitlength}{4144sp}%
\begingroup\makeatletter\ifx\SetFigFont\undefined%
\gdef\SetFigFont#1#2#3#4#5{%
  \reset@font\fontsize{#1}{#2pt}%
  \fontfamily{#3}\fontseries{#4}\fontshape{#5}%
  \selectfont}%
\fi\endgroup%
\begin{picture}(3700,998)(2876,-3710)
\end{picture}%
\caption{\label{fig:examples1} The sets $\example{1}$ and $\example{2}$.}
\end{figure}

As illustrated in Figure \ref{fig:examples1}, we take the sets $\example{1}  = \{0, 1, 2, 3\} \subset \Z $ and $\example{2}  = \{(0, 0), (-1, -1), (-1, 0), (0, 1), (1, 1), (1, 0), (0, -1)\} \subset \Z^2$. The first example gives $X_{\examplep{1}} = \p^1$, with line bundle $\mco_{\example{1}} (1) = \mco (3)$ and the linear system $\linsys{\example{1}}$ consists of all sections. The full sections $\linsysfull{\example{1}}$ are those that do not vanish at the torus fixed points $0$ and $\infty$. The second example $X_{\examplep{2}}$ is a $3$ point blow up of $\p^2$. The bundle $\mco_{\examplep{2}} (1)$ in this case is the anti-canonical bundle and the linear system again consists of all sections.

Many LG models arising in homological mirror symmetry are obtained from pencils on $\mcx_Q$ contained in $\linsys{A}$. It is common for the behavior of these pencils at infinity and zero to be prescribed. We now give a concise definition of this constraint.  
\begin{defn} Let $A^\prime$ be a proper subset of $A$. An $A^\prime$-sharpened pencil on $\mcx_Q$ is a pencil $W  \subset \linsys{A}$ which has a basis $\{s_1, s_\infty\}$ for which $s_1 \in \linsysveryfull{A}$ and $s_\infty = \sum_{\alpha \in A^\prime} c_\alpha s_\alpha$. Let $\shpen{A}{A^\prime}$  be the open subset of $A^\prime$-sharpened pencils in the Grassmannian $Gr_2 (H^0 (\mcx_Q , \mco_Q (1)))$.
\end{defn}
Let us examine two other equivalent ways of defining an $A^\prime$-sharpened pencil. If $s_1 = \sum_{\alpha \in A} c_\alpha s_\alpha \in W \cap \linsysveryfull{A}$, then take $s_0 = \sum_{\alpha \not\in A^\prime} c_\alpha s_\alpha$ and $s_\infty = \sum_{\alpha \in A^\prime} c_\alpha s_\alpha$. The pair $(s_0, s_\infty) \in \C^{(A^\prime)^\circ} \times \C^{A^\prime}$ gives another basis for $W$ which is unique up to a multiple $(\lambda s_0 , \lambda s_\infty)$ for some $\lambda \in \C^*$. We define
\begin{equation*} \mathbf{w} = [s_0 : s_\infty] : \mcx_Q - \{ s_\infty = 0\} \to \C \end{equation*}
to be the Landau-Ginzburg model of the $A^\prime$-sharpened pencil $W$.

Alternatively, we may write $W \in \shpen{A}{A^\prime}$ as the closure of an equivariant map, or orbit, $i_W : \C^* \to \linsysveryfull{A}$. Taking the one parameter subgroup $G_{A^\prime} \subset (\C^*)^A$ given by the cocharacter $\gamma_{A^\prime} = \sum_{\alpha \in A^\prime} e^\vee_\alpha \in (\Z^A)^\vee$ and any very full section $s \in W$, observe that $W = \overline{\{\lambda \cdot s : \lambda \in G_{A^\prime}\}}$. When referring to an $A^\prime$-sharpened pencil, we may utilize any one of these three equivalent viewpoints. As we will observe in the next section, the orbit perspective turns out to be quite useful.

In general, the fibers of $\mathbf{w}$ over $0$ and $\infty$ have bad behavior which is corrected by judicious blow ups. We explain this bad behavior from a global perspective. Let $\mcd_Q = \sum \mcd_i$, where the sum is over the facets of $Q$, be the toric boundary of $\mcx_Q$  and for any subset $J$ of facets, let $\mcz_J = \cap_{j \in J} \mcd_j$. If  $s \in \linsys{A}$, write $\mcy_s$ for the hypersurface defined by $s$ and $\mcy_{s, J} = \mcy_s \cap \mcz_J$. For any subset $U \subset \linsys{A}$, we have the incidence stacks $\mci (U) \subset U \times \mcx_Q$ and $\mci_J \subset U \times \mcx_Q$ whose points are given by pairs $\{(s, y) : s \in U, y \in \mcy_s\}$ and $\{(s, y ) : s \in U, y \in \mcy_{s, J} \}$ respectively. 
\begin{prop} The set $U = \linsysfull{A}$ is the maximal open subset of $\linsys{A}$ for which the projection $\pi_{\linsys{A}} : \mci_J (U) \to U$ is flat for all subsets $J$.
\end{prop}
This follows from the observation that the sections which are not full are equivalent to sections that contain fixed points of the toric action. Thus they contain zero dimensional intersections $\mcz_J$. For our purposes, a reasonable moduli space of sections should not exhibit this behavior. In the next subsection, we modify these sections along with their fibers in order to obtain a proper flat family. 

\subsection{The secondary stack}

To remedy the fact that the incidence varieties give a poorly behaved parameter space for the hypersurface, we review the constructions of the secondary and Lafforgue stacks given in \cite{DKK}, where more details can be found. We assume the reader is familiar with material found in \cite{BCS}, \cite{Cox} and \cite{GKZ}. Given $A$ as above, the secondary polytope $\psec{A} \subset \R^A$ is an $(|A| - d - 1)$-dimensional polytope whose faces correspond to regular subdivisions $S = \{(Q_i, A_i) : i \in I\}$ of $A$. The normal fan $\fans{A}$ of $\psec{A}$ can be refined to a fan $\fanlt{A}$ as in \cite{hacking} and \cite{Lafforgue} by considering pairs $( S, Q^\prime )$ of a subdivision $S$ along with a set $Q^\prime$ which is a face of a subdivided polytope $(Q_i , A_i) \in S$. Then a cone $\sigma_{(S, Q^\prime )} $ in $\fanlt{A}$ is defined as all functions $\eta$ on $A$ whose lower convex hull gives the marked subdivision $S$ and whose minimum is achieved on $Q^\prime \cap A$. 
\begin{prop}[\cite{DKK}] If $\Delta^{A} \subset \R^A$ is the unit simplex, then $\fanlt{A}$ is the normal fan of the Minkowski sum $\ptlaf{A} := \psec{A} + \Delta^A \subset \R^A$.
\end{prop}
\begin{figure}
\begin{picture}(0,0)%
\includegraphics{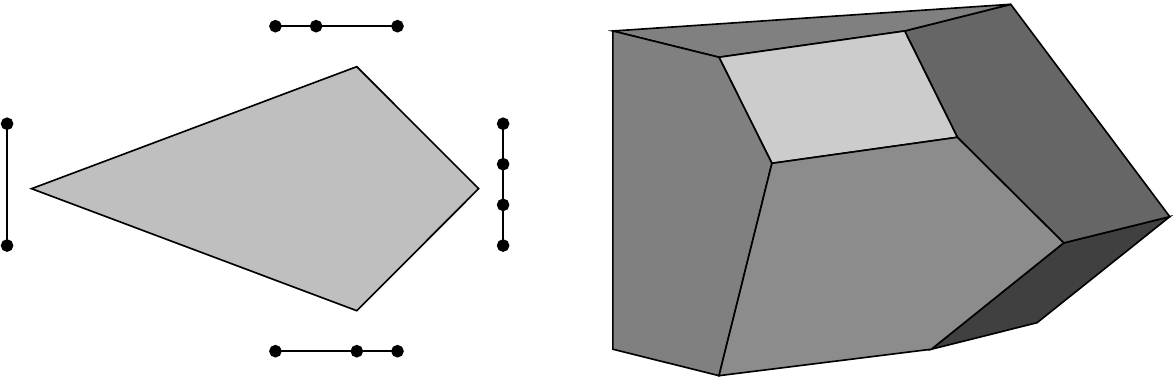}%
\end{picture}%
\setlength{\unitlength}{4144sp}%
\begingroup\makeatletter\ifx\SetFigFont\undefined%
\gdef\SetFigFont#1#2#3#4#5{%
  \reset@font\fontsize{#1}{#2pt}%
  \fontfamily{#3}\fontseries{#4}\fontshape{#5}%
  \selectfont}%
\fi\endgroup%
\begin{picture}(6800,2214)(278,-3223)
\end{picture}%
\caption{\label{fig:example2} The secondary polytope and Lafforgue polytope for $\example{1}$}
\end{figure}
As the secondary polytope $\psec{A}$ lies in a translation of the kernel $L_A \otimes \R := \ker (\R^A \to M_\R \oplus \R)$, it enjoys $(d + 1)$ constraints. Similarly, the polytope $\ptlaf{A}$, which we call the Lafforgue polytope, lies in a translation of the hyperplane $H = \{\sum_\alpha c_\alpha e_\alpha : \sum_{\alpha} c_\alpha = 0 \}$ and therefore has $(|A| - 1)$ dimensions. This drastically limits the number of examples of Lafforgue polytopes that one can visualize, but our case $\example{1}$ rendered in Figure \ref{fig:example2} gives an indication of the relationship between the original marked polytope $(Q, A)$, the secondary polytope $\psec{A}$ and Lafforgue polytope $\ptlaf{A}$. Observe that to each vertex of the secondary polytope, there is a regular triangulation of $(Q, A)$ which can be seen as a unique subset of the facets of the Lafforgue polytope. The second example, $\example{2}$ has a $4$-dimensional secondary polytope with $32$ vertices and a $7$-dimensional Lafforgue polytope. Nevertheless, we will be able to use a polytope derived from this data to analyze the minimal model runs of the homological mirror of $\mcx_{\examplep{2}}$ in the last section.

Since $\ptlaf{A}$ is a Minkowski sum, we have maps $\tilde{\pi}_{A} : X_{\ptlaf{A}} \to X_{\psec{A}}$ and $\tilde{\pi}_Q : X_{\ptlaf{A}} \to \p^{|A| - 1}$. If $i: p \hookrightarrow X_{\psec{A}}$ sends a point to the orbit $\orb{S}$ associated to a subdivision $S = \{(Q_i , A_i ) : i \in I\}$, then we may define $X_S$ as the pullback in the fiber square
\begin{equation} \label{eq:fundex}
\begin{CD}
X_S @>j>> X_{\ptlaf{A}} @>{\tilde{\pi}_Q}>> \p^{|A| - 1} \\ @V{\rho_S}VV @V{\tilde{\pi}_A}VV @. \\ p @>i>> X_{\psec{A}} @.
\end{CD} 
\end{equation}
One only needs to trace through the definitions to see that $\tilde{\pi}_Q \circ j$ maps $X_S$ into the union of the images of the toric varieties $X_{Q_i}$ via their $\linsys{A_i}$ maps. On the level of varieties, this gives us a simultaneous degeneration of $X_Q$ and $\mco_Q (1)$. Taking a global section of $\tilde{\pi}_Q^* (\mco (1))$ yields a degeneration of hypersurfaces. In this way, we have a universal space for performing the degenerations along the lines of the Mumford construction.

We would like to promote this setup to a morphism of stacks $\pi : \laf{A} \to \secon{A}$ so that $\mci (\linsysfull{A} )$ has an \'etale map to $\laf{A}$ and the quotient $[\mci (\linsysfull{A}) / (\C^*)^{d + 1}]$ is naturally an open substack of $\secon{A}$. This was done carefully in \cite{DKK} and we review the procedure here. 

Both the secondary and Lafforgue polytopes have vertices in a hyperplane parallel to $H_\Z = \{\sum c_\alpha e_\alpha : \sum c_\alpha = 0\} \subset \Z^A$. We consider both polytopes to live in $H \subset \R^A$ and write $i_H : H_\Z \to \Z^A$ for the inclusion. As with the case of $\du{A}$, there is an exact sequence
\begin{equation} \label{eq:fundex} \begin{CD} 0 @>>> L_A @>{\delta_A}>> \Z^A @>{\beta_A}>> M @>>> K_A @>>> 0 \end{CD} . \end{equation}
The hyperplanes supporting $\ptlaf{A}$ can be partitioned into horizontal and vertical hyperplane sections $ \du{\ptlaf{A}} = \du{\ptlaf{A}}^h \cup \du{\ptlaf{A}}^v$ and the vertical hyperplanes are all scalar multiples of $(\beta_A \circ i_H)^\vee (\du{A})$. To define a stacky fan, one must choose generators for the one cones. As opposed to taking primitives for the one cones in $\du{\ptlaf{A}}^v$, we take the generators in the image $(\beta_A \circ i_H)^\vee (\du{A})$, while for the horizontal hyperplanes in $\du{\ptlaf{A}}^h$, we choose the primitives of the hyperplanes in $\du{\ptlaf{A}}^h$. This defines the stacky fan which gives Lafforgue stack $\laf{A}$ associated to $A$.

To give the definition of the secondary stack, we rely on a universal colimit construction for toric stacks. It is not hard to show that, if $\tilde{\mcx}_{\psec{A}}$ is the stack given by $\psec{A}$, then there is a map $g: \laf{A} \to \tilde{\mcx}_{\psec{A}}$. The colimit stack $\secon{A}$ of $g$ comes equipped with a map $\pi : \laf{A} \to \secon{A}$. Both $\secon{A}$ and $ \pi$ can be described by the universal property that if $g$ can be factored into two flat, equivariant morphisms $h_1 \circ h_2$ where $h_1 : \laf{A} \to \mcx$ and $\mcx$ is a (good) toric stack, then there is a map $k: \secon{A} \to \mcx$ with $h_1 = k \circ \pi$. This property makes $\secon{A}$ the best toric candidate for the moduli stack of hypersurfaces in $\mcx_Q$. 
\begin{thm}[\cite{DKK}] There is a hypersurface $\hyp{A} \subset \laf{A}$ for which the map $\pi: \hyp{A} \cap (\partial \laf{A})_J \to \secon{A}$ is flat for all horizontal boundary strata $J \subset \du{\ptlaf{A}}^h$. The stack $\secon{A}$ contains a dense open substack $V \approx [\mci (\linsysfull{A} ) / (\C^*)^{d + 1}]$ for which $\pi : \hyp{A} (V) \to V$ is equivalent to the quotient map of $[\mci (\linsysfull{A} ) / \C^*]$.
\end{thm}
This theorem shows that $\pi : \hyp{A} \to \secon{A}$ is a reasonable compactification of the universal hypersurface over the moduli stack of toric hypersurfaces. It is this compactification that allows us to degenerate LG models and understand their components. 

It will be useful to identify the hypersurface of sections in $V \subset \secon{A}$ that do not transversely intersect the toric boundary of $\mcx_Q$. Recall that the principal $A$-determinant from \cite{GKZ} does this precisely and has the secondary polytope as its Newton polytope. Thus it can be written naturally as a section of $\mco_{\secon{A}} (1)$ and we write $\mce_A$ for its zero loci.

\subsection{The stack of Landau-Ginzburg models}

In this subsection we will review a toric compactification of the space of Landau-Ginzburg models arising from $A^\prime$-sharpened pencils. Near the fixed points of this compactification, we give a procedure for obtaining a semi-orthogonal decomposition of the directed Fukaya category of the model. 

The geometry of a fiber polytope has already proven useful in the case of a secondary polytope. As it turns out, this more general  notion works well in describing several moduli problems in the toric setting \cite{logstable}, \cite{KerrFP}. In particular, given two toric varieties, or stacks, $\mcx_{Q_1}$ and $\mcx_{Q_2}$ arising from marked polytopes, one may define a space of equivariant morphisms $\psi : \mcx_{Q_1} \to \mcx_{Q_2}$ for which $\psi^* (\mco_{Q_2} (1)) = \mco_{Q_1} (1)$ up to toric equivalence. This space has a reasonable compactification to a toric stack whose moment polytope is the fiber polytope $\Sigma (Q_2, Q_1)$. 

In the previous section, we examined the case where $Q_2$ was the simplex and $Q = Q_1$. This gave the secondary polytope $\psec{A} = \Sigma (Q_2, Q_1)$ as the moment polytope of the stack $\secon{A}$, which was regarded as a compactification of the moduli stack of toric hypersurfaces of $\mco_Q (1)$ in $\mcx_{Q}$. Prior to this construction, we considered LG models on $Q$ to be $A^\prime$-sharpened pencils $W$ which were given as equivariant maps $i_W : \C^* \to \linsysfull{A}$. Two $A^\prime$-sharpened pencils $W$, $\tilde{W}$ are equivalent if $W = \lambda \tilde{W}$ for some $\lambda \in (\C^*)^{d + 1}$. Thus, from the perspective of equivariant maps, up to toric equivalence a LG model is an equivariant map $\iota_W : \C^* \to [ \linsysfull{A} / (\C^*)^{d + 1}]$.  By equivariant, we mean with respect to the torus embedding $\C^* \to L_A^\vee \otimes \C^*$ given by the cocharacter $\gamma := \delta^\vee_A (\gamma_{A^\prime})$ where $\delta_A$ is defined in equation \ref{eq:fundex} and $\gamma_{A^\prime} \in (\Z^A)^\vee$ in section \ref{sec:tslg}. Note that the codomain of $\iota_W$ is an open chart for the stack $\secon{A}$ implying that the collection of such maps is an open chart of equivariant maps from $\p^1$ to the toric stack $\secon{A}$ with respect to the character map $\gamma : L_A \to \Z$. 

This map $\gamma : L_A \to \Z$ induces a map on polytopes $\psec{A} \to [0, N]$ for some $N$ determined by $A^\prime$. The fiber polytope $\Sigma (\psec{A} , [0, N])$ is known as the monotone path polytope (an example of an iterated polytope, see \cite{BS2}) and is denoted $\Sigma_{\gamma} (\psec{A} )$. The associated fiber stack $\mlg{A}{A^\prime}$ defined in \cite{KerrFP} then serves as a compactification of the open set of $G_{\gamma_{A^\prime}} := \Z \langle \gamma \rangle \otimes \C^*$ orbits contained in the dense subset of $\secon{A}$. Its coarse space is simply the toric variety associated to $\Sigma_\gamma (\psec{A})$.  We codify these notions in the following proposition.

\begin{prop} The quotient stack $\qshpen{A}{A^\prime} = [\shpen{A}{A^\prime} / (\C^*)^{d + 1}]$ of $A^\prime$-sharpened pencils forms an open dense subset of the proper toric stack $\mlg{A}{A^\prime}$. The fixed points of $\mlg{A}{A^\prime}$ are in one to one correspondence with parametric simplex paths of $\gamma : \psec{A} \to [0, N]$ and will be called maximal degenerations of $\mathbf{w}$. 
\end{prop}

Recall from \cite{BS1} that a parametric simplex path of a linear function on a polytope is an edge path which increases relative to the linear function. One consequence of the above construction is that, over any point $\psi$ in $\mlg{A}{A^\prime}$, there is a chain $\langle \psi_1, \ldots, \psi_t \rangle$ of projective lines which has a flat family of degenerated toric varieties (or stacks) lying over it. In the dense orbit, there is one such line, and the toric variety is irreducible, so we obtain a LG model. As we approach the toric boundary, we bubble $\psi$ into a stable map $\{\psi_i\}$ on several components $\cup_1^t \p^1$, and simultaneously degenerate the fibers of the LG model. In a maximally degenerated LG model, we have a chain $\langle C_1, \ldots, C_k \rangle$ of maps to one dimensional orbits of $\secon{A}$. Such strata correspond to the edges of the secondary polytope $\psec{A}$ which in turn correspond to circuit modifications or bistellar flips of the triangulations at the vertices. In \cite{DKK}, components of the fibers over each stable component were examined and found to reproduce well known relations in the mapping class group. They were also conjectured to represent homological mirrors to birational maps of the minimal model on $\mcx_Q^{mir}$.

\subsection{Semi-orthogonal decompositions}

Our next goal is to stratify our space of Landau-Ginzburg models so that for every strata, we obtain a semi-orthogonal decomposition of the Fukaya-Seidel category of the associated model. The decomposition we obtain will bear a direct relationship to the monotone paths corresponding to the the maximal degenerations. To do this, we start by recalling the notion of a radar screen which will yield a class of distinguished basis of paths for the LG model \cite{SeidelFPL}. The definition given here differs from that in \cite{DKK}, but gives the generalizes it and has the advantage of being defined for a generic LG model. Before we start the definition, it is worth keeping in mind that radar screens are auxiliary concepts depending only on configurations of points in $\C^*$ and do not depend on any of the toric stack definitions given earlier. In fact, one can consider their definition to be a logarithmic variant of the more conventional procedure which chooses a distinguished basis of paths to be those with constant imaginary value in the positive real direction \cite{HV}.

Let $E_r = (\C^*)^r / \mathfrak{G}_r$ be the parameter space of $r$ unmarked points in $\C^*$ and $P = \{z_1, \ldots, z_r\} \in E_r$. We order the points so that $ |z_i | \geq  |z_{i + 1}|$ for $1 \leq i \leq r$ and choose a lift $\tilde{P} = \{w_1, \ldots, w_r\}$ such that $e^{w_i} = z_i$. Inductively define paths $p_i : [0,\infty) \to \C$ starting at $w_{i}$ as follows. For $i = 1$, we take the path $p_1 (t) = w_1 + t$. Assume $p_{i }$ has been defined, then we take $p_{i + 1} $ to be the concatenation $p_{i} \ast \ell_i \ast \ell_i^\prime$ where $\ell_i^\prime (t) = w_{i} + t \cdot \text{Im} (w_{i +1} - w_i )$ for $t \in [0, 1]$ and $\ell_i (t) = \text{Re} (w_i) + \text{Im} (w_{i + 1}) - t \cdot \text{Re} (w_i - w_{i + 1})$. While $p_i$ are a collection of overlapping paths, it is clear that for any $\varepsilon$, we can perturb $p_i$ to $\tilde{p}_i$ so that $||p_i - \tilde{p}_i||_{L^2} < \varepsilon$ and $\{\tilde{p}_i: 1 \leq i \leq r \}$ forms a distinguished basis for $\tilde{P}$. Furthermore, if the values $|z_i|$ are distinct, the distinguished basis defined in this way is unique up to isotopy for $\varepsilon \ll 1$. 

\begin{defn} With the notation above, we say that $\mcb_{\tilde{P}} = \{e^{\tilde{p}_i} : 1 \leq i \leq r \}$ is a radar screen distinguished basis and take $\mcr_{P}$ to be the collection of all such bases. If $\tilde{P} \subset \{w \in \C : 0 \leq  Im (w) < 2\pi\}$ we write $\mcb_{P}$ and call any such basis a fundamental radar screen.
\end{defn}

Our main application of this definition is when $P$ is the collection of critical values of a LG model $\mathbf{w} \in \shpen{A}{A^\prime}$. Let $\Delta_{A, A^\prime}$ be the variety of all $A^\prime$-sharpened pencils that do not intersect the principal $A$ determinant $\mce_A$ transversely regarded as a subvariety of $\shpen{A}{A^\prime}, \qshpen{A}{A^\prime}$ or $\mlg{A}{A^\prime}$. We denote its complement in $\shpen{A}{A^\prime}$ and $\qshpen{A}{A^\prime}$ by $V_{A, A^\prime}$ and $\mcv_{A, A^\prime}$ respectively. Take $\tilde{E}_r = E_r / \C^*$ to be the quotient where $\C^*$ acts by multiplication. 

\begin{prop}  Suppose $r = |i_W^{-1} (\mce_A )|$ for some $W \in V_{A, A^\prime}$. The map $\mathbf{c} : V_{A, A^\prime} \to E_r$ given by $\mathbf{c} (W) = i_W^{-1} (\mce_A)$ can be completed to a commutative diagram
\begin{equation*}
\begin{CD} V_{A, A^\prime} @>{\mathbf{c}}>> E_r \\ 
@VVV @VVV \\
\mcv_{A, A^\prime} @>{\tilde{\mathbf{c}}}>> \tilde{E}_r
\end{CD}
\end{equation*}
\end{prop}

\begin{proof} This follows from the quasi-homogeneous property of the principal $A$ determinant with respect to the $(\C^*)^{d + 1}$ action on $\C^A$ (\cite{GKZ}). Indeed, we have that if $i_W, i_{\tilde{W}} \in \shpen{A}{A^\prime}$ are equivalent, then there exists $(\lambda, \eta) \in (\C^*)^{d + 1} = \C^* \times (\C^* \otimes N) $ such that $\lambda (1 \otimes \beta_A)^\vee (\eta ) i_W (z) = i_{\tilde{W}} (z)$ for all $z \in \C^*$. But then $E_A (i_{\tilde{W}} (z) ) = 0$ if and only if $\lambda^v \cdot \beta_A (p)(\eta ) \cdot E_A (i_W (z)) = 0$ where $v = (d + 1) \text{Vol} (Q)$, $p \in \psec{A}$ and $M$ is identified with $\Hom (\C^* \otimes N , \C^*)$.
\end{proof}

This proposition shows that a choice of radar screen for the critical values of an $A^\prime$-sharpened pencil can be consistently made on the quotient space $\mcv_{A, A^\prime}$. Now, the discriminant $\Delta : E_r \to \C$ given by $\Delta (z_1,\ldots, z_r) = \prod_{i < j} (z_i - z_j)^2$ is homogeneous and thus its zero locus is pulled back from $\tilde{E}_r$. The associated braid group $\tilde{B}_r = \pi_1 (\tilde{E}_r - \{\Delta = 0\} )$ is in fact a quotient of the subgroup of the braid group $B_{r + 1}$ which is pure on the strand at the origin. It is clear that the map $\tilde{\mathbf{c}}$ induces a representation of the fundamental group of $\mcv_{A, A^\prime}$ into $\tilde{B}_r$. More generally,  we have a representation of fundamental groupoids
\begin{equation*} \mathbf{r} : \Pi (\mcv_{A, A^\prime} ) \to \Pi (\tilde{E}_r - \Delta ). \end{equation*}

Define $\Delta_\R (z_1 , \ldots, z_r ) = \prod_{i < j} (|z_i | - |z_j|)^2$ to obtain a real stratification $\tilde{\mcs} = \{R_\rho : \rho \in \mcp\}$ of $\tilde{E}_r$. Here $\mcp$ denotes the set of partitions of $\{1, \ldots, r\}$ and $R_\rho = \{\{z_1, \ldots, z_r \} : |z_i| = |z_j| \text{ for } i \sim_\rho j ,\text{ and } |z_i| \leq |z_{i + 1}|\}$.

\begin{defn} The pullback stratification 
\begin{equation*} \mcs = \{R \text{ a component of } \tilde{\mathbf{c}}^{-1} (R_\rho ) : R_\rho \in \tilde{\mcs}\} \end{equation*}
on $\qshpen{A}{A^\prime}$ will be called the norm stratification.
\end{defn}

From the description of the toric boundary strata of $\mlg{A}{A^\prime}$, we may extend the norm stratification on $\mcv_{A, A^\prime}$ to the boundary. We will avoid the details of this extension, which are evident from the fact that the orbits degenerate to sequences of orbits, and refer to the resulting stratification on $\mlg{A}{A^\prime}$ as the extended norm stratification.

To use the definitions given above, we need a result that gives us a well defined category on which to work. If the sharpening set $A^\prime$ is chosen carefully, the associated LG model $\mathbf{w}$ has a sensible definition of a Fukaya-Seidel category $\fs{\mathbf{w}}$. For example, we have the following proposition.

\begin{prop}[\cite{DKK}] \label{prop:lef} Assume $A^\prime \subset A$ is contained in the interior of $Q$. If $W$ transversely intersects the principal $A$-determinant, then its LG model $\mathbf{w}$ is a Lefschetz pencil.
\end{prop}

In particular, the singularities are isolated and Morse, and parallel transport is well defined along the base so that the usual notion of Fukaya-Seidel categories applies \cite{SeidelFPL}. Along with this, we have that the collection of distinguished bases is acted on by the full braid group $B_r$ which extends to an action on exceptional collections via mutations \cite{vcm}. 

From this proposition and the results in the above references, it is not difficult to obtain our main theorem for this section.

\begin{thm}\label{thm:so} The space $\mcv_{A, A^\prime}$ has a stratification $\mcs$ such that, for any component $R \in \mcs$, there exists a $\Z^r$ torsor $\mcc_R$ of semi-orthogonal decompositions of $\fs{\mathbf{w}}$ satisfying:
\begin{itemize} 
\item[(i)] If $R_1 < R_2 \in \mcs$ then there is a bijective map $\tau :\mcc_2 \to \mcc_1$ such that the semi-orthogonal decomposition $S \in \mcc_2$ refines that given by $\tau (S)$. 

\item[(ii)] If $\gamma: [0,1] \to \mcv_{A, A^\prime}$ is any morphism in $\Pi (\mcv_{A, A^\prime} )$ with $\gamma (0), \gamma (1) \in R_0$ giving exceptional collections, then $\mathbf{r} (\gamma )$ acts by mutation to map $\mcc_{R_0}$ to itself.

\item[(iii)] Assume $R \in \mcs$ and $\mathbf{w} \in \overline{R}$ is in the boundary of $\mlg{A}{A^\prime}$ and corresponds to a sequence $\langle \mathbf{w}_1 , \ldots, \mathbf{w}_t \rangle$. Then every semi-orthogonal decomposition in $\mcc_R$ refines the  decomposition $\langle \fs{\mathbf{w}_1} , \ldots , \fs{\mathbf{w}_t} \rangle$.
\end{itemize}
\end{thm}

The first two statements follow from directly from the definition, while the third from the structure of the monotone path polytope. Note that this gives a direct relationship between the combinatorics of maximal degenerations, or parametric simplex paths, and decompositions of Fukaya-Seidel categories of Landau-Ginzburg models near such degenerations.

\section{$A_n$-categories}

In this section we consider the most basic possible case, the directed $A_n$-category. We give a detailed construction of the secondary, Lafforgue and monotone path stacks in this case. In particular, we describe the combinatorics of the monodromy maps around the discriminant and toric boundary. Symplectic geometry in these cases is completely absent, as the Fukaya categories are more of a combinatorial nature. Nevertheless, the structure and geometry of the decompositions and representations of this category is surprisingly rich and illustrates some of the techniques that are applied in higher dimensions.


\subsection{The Lafforgue stack of an interval}

The derived $A_n$ category can be given by the Fukaya-Seidel category of a single polynomial 
\begin{equation*} \mathbf{w} (x) = c_{n + 1} x^{n + 1} + \cdots + c_1 x + c_0 . \end{equation*}
whose marked Newton polytope $(Q, A)$ is clearly $Q = [0, n + 1]$, $A = \{0, 1, \ldots, n + 1\}$. 

A first step to understanding this example is to characterize the stacks $\secon{A}$ and $\laf{A}$. We will derive $\laf{A}$ by obtaining its stacky fan. The secondary polytope of $A$ was examined in \cite{GKZ} and seen to be affinely equivalent to the representation theoretic polytope $P (2 \rho )$ which is the convex hull of the dominant weights $\{\omega\}$ of $A_n$ such that $\omega \leq 2\rho$ where $2\rho $ is the sum of the positive roots. We begin by reviewing their observations and establishing notation.

Take $\{e_0, \ldots, e_{n + 1} \}$ as a basis for $\Z^A$ and $\{\alpha_1, \ldots, \alpha_n\}$ a basis for $\Lambda_r \approx \Z^n$ and examine the fundamental exact sequence for $A$
\begin{equation*}
\begin{CD} 0 @>>> \Lambda_r @>{\delta_A}>> \Z^A @>{\beta_A}>> \Z^2 @>>> 0 \end{CD}
\end{equation*}
where $\beta_A (e_i) = (1, i)$ and $\delta_A (\alpha_i ) = -e_{i  - 1} + 2 e_i - e_{i + 1} $. Write $C_n$ for the $A_n$ Cartan matrix 
%
%
and recall that this serves as a transformation matrix from the simple roots to the fundamental weights. We will view $\Lambda_r$ as the $A_n$ root lattice with simple roots $\{ \alpha_i\}$, fundamental weights $\{\lambda_1 , \ldots, \lambda_n\} \subset \Lambda := \Lambda_r^\vee$ and $\rho = \sum_{i = 1}^n \lambda_i = \frac{1}{2} \sum_{i = 1}^n i (n + 1 - i) \alpha_i$ the Weyl element.

\begin{prop}[\cite{GKZ}] \label{prop:secfan} The normal fan $\mcf$ of $\psec{A}$ in $\Lambda \otimes \R$ has $1$-cone generators $\{\alpha_1, \ldots, \alpha_n , -\lambda_1, \ldots, -\lambda_n\}$ in the weight lattice $\Lambda$ and cones
\begin{equation*} \sigma_{I, J} = \Span_{\R_{\geq 0}} \{ -\lambda_i , \alpha_j : i \in I, j \in J\} \end{equation*}
for any pair of disjoint sets $I, J \subset [n]$.
\end{prop}
The vertices of $\psec{A}$ are easily seen to be in one to one correspondence with subsets $K = \{k_0 <  \ldots < k_m\} \subset \{1, \ldots, n\}$ representing the triangulations $T_K = \{([k_i , k_{i + 1}], \{k_i, k_{i + 1}\})\}$ and corresponding to the vertex $\varphi_K = \sum_{j = 0}^m (k_{i + 1} - k_{i - 1} ) e_{k_i}$ of $\psec{A}$. 

In order to obtain the secondary stack, we need to write out the stacky fan for the Lafforgue stack and find the limit stack. The facets of the Lafforgue polytope were shown in \cite{DKK} to correspond to pointed coarse subdivisions of $A$. Since $\ptlaf{A}$ is an $(|A| - 1)$ dimensional polytope, we define the supporting hyperplane functions $\du{\ptlaf{A}}$ as elements in $(\Z^A)^\vee$ but restrict them to linear functions on $ \Gamma = \{\sum c_i e_i : \sum c_i = 0\}$. Letting $f_i = \delta_A (\alpha_i)$ and $f_0 =  e_{0} - e_{1}$, we take $\{f_0, f_1, \ldots, f_n\}$ as a basis for $\Gamma$ so that $\delta_A : \Lambda \to \Z^A$ lifts to $\tilde{\delta}_A : \Lambda \to \Gamma$. As we will show in a moment, there are $3n + 2$ facets of $\ptlaf{A}$, so the stacky fan is obtained by a fan in $\R^{3n + 2}$ along with a map $\xi_A : \Z^{3n + 2} \to \Gamma^\vee$ which gives the group $\tgroup{\ptlaf{A}} \simeq \ker (\xi_A ) \otimes \C^*$. We write $\{g_i : 1 \leq i \leq 3n + 2\}$ for the standard basis of $\Z^{3n + 2}$

The pointed coarse subdivisions $(S, B)$ of $A$ can be classified into three types. For each $1 \leq i \leq n$, there is a pointed subdivision $( (Q, A - \{i\}), A - \{i \} )$ whose supporting hyperplane function $g_i$ is given by $e^\vee_i \in (\Z^A)^\vee$ so that $\xi_A (g_i) = - f_{i - 1} + 2f_i - f_{i + 1}$. Also, for every $1 \leq i \leq n$, there are two pointed subdivisions corresponding to $Q = [0, i] \cup [i , n + 1 ]$ with pointing set $\{0, \ldots, i\}$ and $\{i , \ldots , {n + 1} \}$ respectively. The supporting primitives are easily seen to be $g_{2i} = \sum_{j = i + 1}^{m + 1} (j - i) e^\vee_j$ and $g_{3i} = \sum_{j = 0 }^i (i - j) e^\vee_j$ implying $\xi_A (g_{2i} ) = -f_i $ and $\xi_A (g_{3i} ) = -f_i - f_0$. Finally, there are two vertical pointed subdivisions $( (Q, A), \{ 0 \})$ and $( (Q, A ), \{{n + 1}\})$ corresponding to the one cones for $\mcx_Q$. The linear functions corresponding to these two subdivisions are $g_{3n + 1}, g_{3n + 2}$ which map to $\xi_A (g_{3n + i}) = (-1)^{i + 1} f_0^\vee$. We can write the map $\xi_A$ as the matrix
\begin{equation*}
\xi_A = \left[ \begin{array}{c| c |c| cc }
-1 \text{ } 0  \text{ }  \cdots  \text{ } 0 & 0  \text{ }  \cdots  \text{ }  0 & 1 \text{ } \cdots  \text{ } 1 & 1 & -1 \\
\hline 
 \text{ }  C_n  \text{ }  &   \text{ }  -I \text{ }   & \text{ } -I  \text{ }  & 0  & 0 
\end{array}
\right]
\end{equation*}
where $I$ is the $n \times n$ identity matrix. 

The maximal cones of the Lafforgue fan $\mcf_{\ptlaf{A}}$ are indexed by pointed triangulations $\{(K, k) :K \subset [n], k \in K \}$. For example, if $k \ne 0, n + 1$, the cone in $\mcf_{\ptlaf{A}}$ associated to $(K, k)$ is
\begin{equation*} \sigma_{(K, k)} = \cone ( \{g_j : j \not\in K\} \cup \{g_{2j} : j \in K, j \leq k\} \cup \{ g_{3j} : j \in K , j \geq k \} ).\end{equation*}
If $k \in \{0, n + 1\}$, we add $g_{3n + 1}, g_{3n + 2}$ respectively to the generating set above. This in particular implies that  $\mcf_{\ptlaf{A}}$ is a simplicial fan. 

There is a map of fans $\pi_{\mcf} : \mcf_{\ptlaf{A}} \to \mcf_{\psec{A}}$ which can be promoted to a map of canonical toric stacks. Performing a calculation of the limit stacky fan defined in \cite{DKK} then gives the following proposition. 

\begin{prop} The Lafforgue stack $\laf{A}$ for $A = \{0, \ldots, n + 1 \}$ is smooth with covering $\{U_{(K,k) }\}$. The secondary stack $\secon{A}$ is smooth and is given by the stacky fan given in Proposition \ref{prop:secfan}.
\end{prop}

This implies that for $n > 1$, the secondary stack is given by taking the canonical stack of the normal fan $\mcf_{\psec{A}}$ while for $n = 1$ we have $\secon{A} = \p (1, 2)$. This reproduces the stacks studied in \cite{blume} which are quotients of the Losev-Manin stack. In all cases, there is a covering of $\secon{A}$ by $\{U_{K} \}_{K \subset [n]}$ where $U_K$ is the chart associated to the cone $\sigma_{K, [n] - K}$. Let us describe an open chart $U_K$ in the covering given above.  For $K = \{k_1 ,  \ldots , k_m \} \subset [n]$, assume $k_1 < \cdots < k_m$ and write $k_0 = 0$, $k_{m + 1} = n + 1$ and $r_i = k_i - k_{i -1}$ for $i = 1, \ldots, m + 1$. Write $\mu_r$ for the group of $r$-th roots of unity and let $\mu_r$ act on $\C^{r }$ via $\zeta (z_1, \ldots, z_{r }) = (\zeta z_1 , \zeta^2 z_2 , \ldots, \zeta^{r -1 } z_{r - 1} , z_r )$. We also take this as an action on the first $(r - 1)$ coordinates. Then, using the basis of the open cones in the weight lattice given in \ref{prop:secfan}, the open stack $U_{K}$ is easily seen to be the quotient stack 
\begin{equation*}
U_K \approx \left[ \C^{r_1} \times \cdots \times \C^{r_{m + 1} -1} / \mu_{r_1} \times \cdots \times \mu_{r_{k + 1}} \right].
\end{equation*}
This local description extends over the Lafforgue stack and the universal hypersurface. Indeed, writing $G_K$ for the group $\mu_{r_1} \times \mu_{r_{m + 1}}$, it is not hard to show that there is a polydisc neighborhood $V_K = [D_1 \times \cdots \times D_m / G_K ]$ near the origin of $U_K$, with
\begin{equation*} \pi_{\mch}^{-1} (V_K ) \approx \left[ \left( \cup_{j = 1}^{m + 1} \mu_{r_j} \right) \times D_1 \times \cdots \times D_m / G_K \right] . \end{equation*} 
Here $G_K$ acts in the obvious way on the set $\cup \mu_{r_j}$.

\subsection{Vanishing trees of maximal degenerations}

Having described the secondary stack and Lafforgue stack for $A_n$, we would like to consider our space of Landau-Ginzburg models which define the directed $A_n$-category. For this purpose, we choose $A^\prime = \{0\}$ and consider all $A^\prime$-sharpened pencils on $\mcx_Q = \p^1$. Since $A^\prime$ is not in the interior of $A$, we cannot apply Proposition \ref{prop:lef}. However, in this case we can state the following proposition whose proof is evident from the definitions in section \ref{sec:toric}.

\begin{prop} Let $A = \{0, \ldots, n + 1\}$ and $A^\prime = \{0\}$. The Landau-Ginzburg model of an $A^\prime$-sharpened pencil is a degree $(n + 1)$ polynomial $\mathbf{w}(z) = c_{n + 1} z^{n + 1} + \cdots + c_{0}$ such that $c_i \ne 0$ for $0 \leq i \leq n  + 1$. 
\end{prop}

Note that the Fukaya-Seidel category is extremely sensitive to the choice of sharpening point (or set). For example, if $A = \{0, 1, 2\}$ and  we chose $A^\prime = \{1\}$ instead of $\{0\}$, we would obtain the homological mirror of $\p^1$  instead of the category of vector spaces (or the $A_1$-category). 

Recall that the moment polytope of the stack of $\mlg{A}{A^\prime}$  is the monotone path polytope of $\psec{A}$ relative to the function $\gamma_{A^\prime} = e_0^\vee$ and that maximal degenerated LG models correspond to monotone edge paths of $\psec{A}$. We describe this polytope in the following proposition.

\begin{prop}\label{prop:Anmon} The monotone path polytope $\Sigma_{\gamma_{A^\prime}} (\psec{A} )$ is combinatorially equivalent to an $(n - 1)$-dimensional cube.
\end{prop} 

\begin{proof} By the results of \cite{BS1}, the vertices of $\Sigma_{\gamma_{A^\prime}} (\psec{A} )$ correspond to parametric simplex paths on $\psec{A}$. We recall that the vertices of $\psec{A}$ are labeled by subsets $K = \{k_0, \ldots, k_m\} \subset \{1, \ldots, n\}$ with associated triangulation $T_K = \{[k_i, k_{i + 1}]\}$. The image of $\gamma_{A^\prime}$ is easily seen to be $[1, n + 1]$, where the set of vertices of $\psec{A}$ sent to $1$ are all subdivisions $\{1,k_1 , \cdots, k_m\}$. Omitting the element $1$, we identify these with subsets $J = \{k_1, \ldots, k_m\} \subset \{2, \ldots, n\}$. Now observe that to any such vertex, there is a unique parametric simplex path on $\psec{A}$ relative to $\gamma_{A^\prime}$ which has $J$ as its minimum. Indeed, if $P = (J = K_0, K_1, \ldots, K_r)$ is a sequence of vertices in a parametric simplex path, then $\{K_i, K_{i + 1}\}$ is an edge of $\psec{A}$ and $\gamma_{A^\prime} (K_i) < \gamma_{A^\prime} (K_{i + 1})$. It is not hard to see that $K_i = \{k_{i + 1} , \ldots, k_m\}$ gives such a path, establishing the existence claim. To see that it is unique, suppose $P^\prime = (K_0, \ldots, K_i, K_{i+1}^\prime, \cdots, K_{r})$ is any other parametric simplex path. Since $\{K_i, K_{i + 1}^\prime\}$ is an edge of $\psec{A}$, we have that $K_i^\prime$ is obtained from $K_i$ by inserting or deleting a element. As $\gamma_{A^\prime} (K_{i + 1}^\prime ) > \gamma_{A^\prime} (K_i)$, we cannot insert a point, and deleting any element besides $k_{i + 1}$ does not affect the value of $\gamma_{A^\prime}$. Therefor $K^\prime_{i + 1} = K_{i + 1}$ and the path is unique. 

By the Minkowski integral description of fiber polytopes, one easily observes that any face of $\gamma^{-1}_{A^\prime} (1)$ gives a face of $\Sigma_{\gamma_{A^\prime}} (\psec{A})$. Since the vertices are in bijection, this implies that the face lattices are equal and yields the proposition.
\end{proof}
From the proof of this proposition we obtain a combinatorial description of the sequence of circuits associated to maximal degenerations. Our next goal is to give a complete description of the semi-orthogonal decompositions connected to such sequences. We first must recall the degeneration and regeneration procedure from \cite{DKK}.  Consider a monotone path specified by $J = \{0, 1 = k_0, k_1, \ldots, k_m = n + 1\}$ and a function $\eta : A \to \Z$ which defines the triangulation given by $J$ (see \cite{BFS}  or \cite{GKZ}). Briefly recall that, for $a < b \in A$, if we denote $\eta^\prime_{a,b} = \eta(b ) - \eta (a) / (b - a)$ then this means that $\eta^\prime_{k_i, k_{i + 1}}$ is increasing relative to $i$ and that $\eta (a)$ lies above the under-graph of $\eta$ for $a \not\in J$. To simplify the treatment, we also assume that $\eta^\prime_{k_i, k_{i + 1}} \in \Z$. Fix any $\mathbf{c} = (c_{n + 1}, \ldots, c_0) \in \C^{n + 2}$ such that $c_i = 1$ for $i \in J$, and define the family of polynomials
\begin{equation*}
\psi (\mathbf{c},s, t) (z) = \left( \sum_{i=1}^{n +1} c_i s^{\eta (i)} z^i \right) + s^{\eta (0)} t 
\end{equation*}
which, for $s \ne 0$ give very full sections $\psi (\mathbf{c},s, t) \in \linsysveryfull{A}$. 
Notice that this gives an $s$ parameterized family of $A^\prime$-sharpened pencils $\psi (\mathbf{c}, s, \_ ) $. After quotienting with the appropriate group, we can think of $\psi$ as a function from $\C^*$ to $\mlg{A}{A^\prime}$, or as a function from $(\C^*)^2 $ to $\secon{A}$. We will shift between perspectives in what follows.

As was seen in the proof of Proposition \ref{prop:Anmon}, the sequence $\langle C_1, \ldots, C_m \rangle$ of circuit modifications  associated to $J$ are supported on  $C_i = \{0, k_{i - 1} , k_{i } \}$ and correspond to edges of $\psec{A}$ with vertices $T_{K_{i -1}}$ and $T_{K_i}$ where $T_i = \{0, k_i , k_{i + 1} , \cdots , k_m\}$. For any $1 \leq i \leq m$, we may reparameterize $\psi$ so as to obtain a regeneration of $C_i$. First recall that such a regeneration of $C_i$ is a map $\tilde{\psi}_i$ completing a diagram
\begin{equation} \label{diag:reg}
\begin{tikzcd} \C^* \arrow[hook]{r}\arrow{d}{\rho_i} & \C \times \C^* \arrow{d}{\tilde{\psi}_i} \\ 
\secon{C_i} \arrow[hook]{r} & \secon{A}
\end{tikzcd}
\end{equation}
where $\rho_i$ is \'etale onto the complement of $\{0, \infty\} \subset \secon{C_i}$, the top arrow is the inclusion into $\{0\} \times \C^*$ and $\tilde{\psi}_i$ is a finite map. Explicitly, this is given by reparameterizing $\tilde{\psi}_i (s, t) = \psi \left( \mathbf{c}, s, s^{\eta (k_{i + 1} ) -\eta (0) - k_{i + 1} \eta^\prime_{k_{i + 1}, k_i}} t \right)$ and completing to $s = 0$. Indeed, letting $z = s^{\eta^\prime_{k_i, k_{i + 1} }} u$ we have that $\lim_{s \to 0} \tilde{\psi}_i (s, t) $ converges to the circuit pencil which can be written in the $u$ coordinate as
\begin{equation} \label{eq:wi} \mathbf{w}_i (u) = u^{k_{i + 1}} + u^{k_i} + t . \end{equation}
The functions $\mathbf{w}_i$ are precisely those yielding the subcategories in the semi-orthogonal decomposition from Theorem \ref{thm:so}. One can think of the reparameterization as giving an asymptotic prescription for the $i$-th bubble in the stable map limit of $\psi$ as $s$ tends to $0$. Moreover, it is important to remember that the fibers of $\pi :\laf{A} \to \secon{A}$ over $\tilde{\psi}$ themselves degenerate into reducible chains of projective lines $\cup_{j = i + 1}^m \p^1$ as in Figure \ref{fig:degeneration}.

\begin{figure}[t]
\begin{picture}(0,0)%
\includegraphics{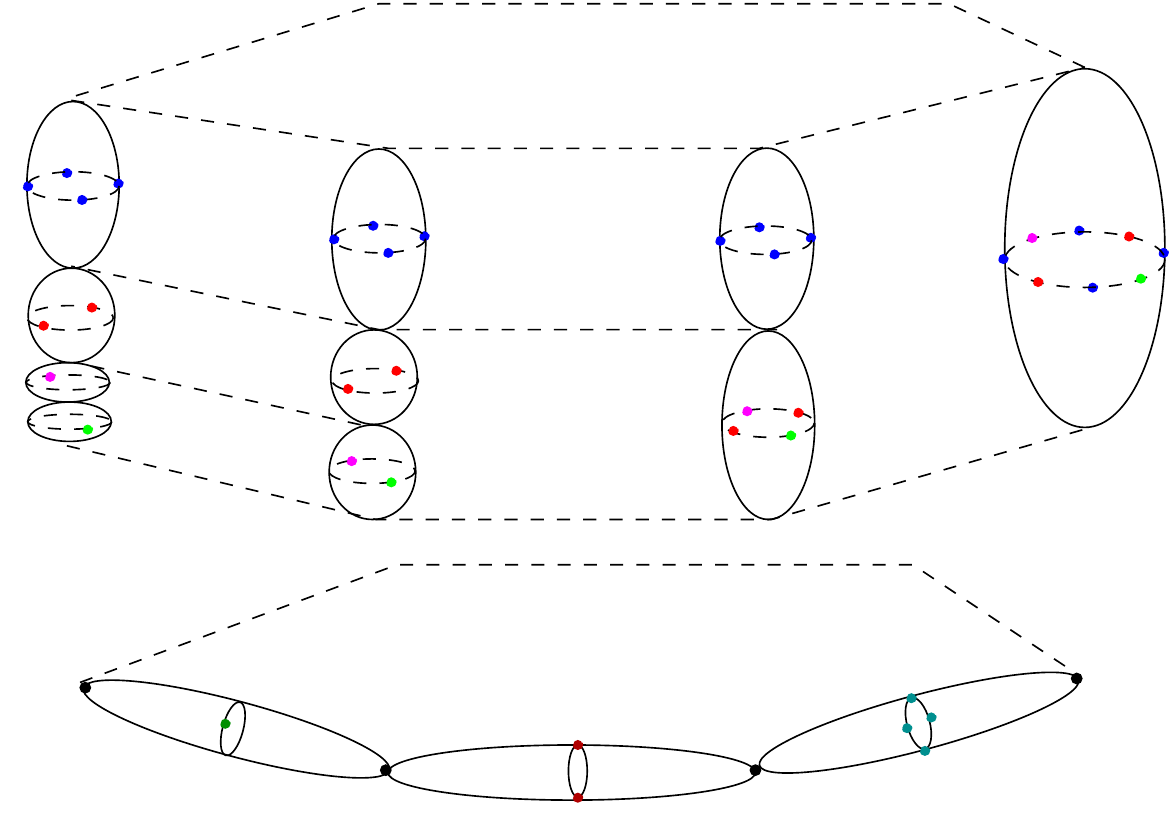}%
\end{picture}%
\setlength{\unitlength}{4144sp}%
\begingroup\makeatletter\ifx\SetFigFont\undefined%
\gdef\SetFigFont#1#2#3#4#5{%
  \reset@font\fontsize{#1}{#2pt}%
  \fontfamily{#3}\fontseries{#4}\fontshape{#5}%
  \selectfont}%
\fi\endgroup%
\begin{picture}(5349,3820)(301,-3599)
\put(1165,-3350){\makebox(0,0)[lb]{\smash{{\SetFigFont{6}{7.2}{\rmdefault}{\mddefault}{\updefault}{\color[rgb]{0,0,0}$C_1$}%
}}}}
\put(2821,-3557){\makebox(0,0)[lb]{\smash{{\SetFigFont{6}{7.2}{\rmdefault}{\mddefault}{\updefault}{\color[rgb]{0,0,0}$C_2$}%
}}}}
\put(4475,-3350){\makebox(0,0)[lb]{\smash{{\SetFigFont{6}{7.2}{\rmdefault}{\mddefault}{\updefault}{\color[rgb]{0,0,0}$C_3$}%
}}}}
\put(2737,-2811){\makebox(0,0)[lb]{\smash{{\SetFigFont{6}{7.2}{\rmdefault}{\mddefault}{\updefault}{\color[rgb]{0,0,0}$\secon{A}$}%
}}}}
\put(2737,-123){\makebox(0,0)[lb]{\smash{{\SetFigFont{6}{7.2}{\rmdefault}{\mddefault}{\updefault}{\color[rgb]{0,0,0}$\laf{A}$}%
}}}}
\put(316,-2750){\makebox(0,0)[lb]{\smash{{\SetFigFont{6}{7.2}{\rmdefault}{\mddefault}{\updefault}{\color[rgb]{0,0,0}$\{0,1,2,4,8\}$}%
}}}}
\put(5145,-2736){\makebox(0,0)[lb]{\smash{{\SetFigFont{6}{7.2}{\rmdefault}{\mddefault}{\updefault}{\color[rgb]{0,0,0}$\{0,8\}$}%
}}}}
\put(3669,-3473){\makebox(0,0)[lb]{\smash{{\SetFigFont{6}{7.2}{\rmdefault}{\mddefault}{\updefault}{\color[rgb]{0,0,0}$\{0,4,8\}$}%
}}}}
\put(1758,-3473){\makebox(0,0)[lb]{\smash{{\SetFigFont{6}{7.2}{\rmdefault}{\mddefault}{\updefault}{\color[rgb]{0,0,0}$\{0,2,4,8\}$}%
}}}}
\end{picture}%

\caption{\label{fig:degeneration} The pullback of $\secon{A}$ and $\laf{A}$ for the monotone path $J = \{0, 1, 2, 4, 8\}$}
\end{figure}

Letting $s$ and $t$ tend to $0$ in $\tilde{\psi}_{i + 1} (s, t)$, one approaches the fixed point of $\secon{A}$ associated to the triangulation $T_{K_i}$ in the monotone path. The $u$ roots $\fib{A}{q} := \pi_{\mch}^{-1} (q)$ of $ q = \tilde{\psi}_{i + 1} (s, t ) \in \secon{A}$ converge to the degenerated hyperplane section. As described at the end of the previous subsection, this hypersurface degeneration results in a partition of the fiber $\fib{A}{q}$ into $(m - i)$ subsets $F_{i,i + 1} (q),  \ldots, F_{i, m} (q)$. For every $j > i $, the set $F_{i, j}(q)$ converges to $(k_j  - k_{j - 1})$ roots of unity of the $j$-th component of the degeneration of the fiber, while the component $F_{i, i + 1} (q)$ converges to the $k_i$ roots of unity. Thus all of the $F_{i,j}(q)$ are cyclically ordered sets. For the example illustrated in Figure \ref{fig:degeneration}, the sets $F_{0,1} (q) , F_{0,2} (q), F_{0,3}(q), F_{0,4}(q)$ are colored green, purple, red and blue, respectively.

For $j > i + 2$ the subsets $F_{i,j} (q)$ experience no monodromy as $q$ varies near the $C_i$ component, regardless of the path. Thus for $j > i + 2 $ there is a collection of unique monodromy isomorphisms $\tau_{i,j} : F_{i, j} (q) \to F_{i + 1,j} (q^\prime )$ where $q$ and $q^\prime$ approach $0$ and $\infty$ respectively of $C_i$. To identify the remaining sets, we must choose a radar screen $\mcb$ for $\tilde{\psi}_{i + 1} (s, t)$ to obtain the isomorphism
\begin{equation*} \tau_{i, i + 2} :F_{i , i + 1} (q) \cup F_{i, i + 2}(q) \to F_{i + 1, i + 2} (q^\prime) . \end{equation*}
To define $\tau_{i, i + 2}$, we take the path first path $p_j \in \mcb$ which does not end on a point in $C_i$, degenerate, and reparameterize the component $\gamma_i$ of $p_j$ so that it is a path from $0$ to $\infty$ in $C_i$  (if $i = 1$, take the last path and concatenate to extend it to $0$). Then $\tau_{i, i + 2}$ is defined as the monodromy along $\gamma_i : [0, 1] \to C_i$.

Note that in the $1$ dimensional case, vanishing thimbles of a polynomial $\mathbf{w}$ are simply paths in $\C$ with endpoints on a fiber $\mathbf{w}^{-1} (q)$. Labeling them according to which path in the radar screen they are defined by, we obtain an edge labeled tree which we refer to as the vanishing tree of $\mathbf{w}$ with respect to $\mcb$. If we omit the grading, this tree encodes all of the data necessary to compute the algebra of the morphisms between exceptional objects in $\fs{\mathbf{w}}$. We would like to give a concrete combinatorial formulation of this vanishing tree. 

Towards this end, suppose $S_1, S_2$ and $S_3$ are finite sets such that $S_1$ and $S_3$ have a cyclic order and $|S_3|= |S_1 | + |S_2|$. We call a bijection $\sigma : S_1 \cup S_2 \to S_3$ a cyclic $|S_2|$ insertion if $\sigma|_{S_1}$ preserves the cyclic order. Now, assume $S_2$ comes equipped with a total order $<$ and label $S_2 = \{s_1, \cdots, s_{|S_2|}\}$. Extend this to a partial order on $S_1 \cup S_2$ by taking $s < s^\prime$ if $s \in S_1$ and $s^\prime \in S_2$. If $\sigma$ is a cyclic $|S_2|$ insertion and $s_l \in S_2$, define $m_\sigma (s_k ) = s^\prime \in S_1 \cup S_2$ to be the unique element less than $s_k$ such that every element $s \in S_3$ in the cyclic interval between $\sigma (m_\sigma (s_k))$ and $ \sigma (s_k ) $ satisfies $s_k < \sigma^{-1} (s)$. We define the incidence graph of this function
\begin{equation*}
I_{\sigma , <} = \{ (\sigma (s) , \sigma (m_{\sigma} (s))) : s \in S_2 \} \subset S_3 \times S_3
\end{equation*}
Now, let $\tilde{P} (s)$ be a choice of logarithms of the critical values of $\psi (\mathbf{c}, s, \_ )$.

\begin{thm} \label{thm:an}  For $s \ll 1$ and a radar screen distinguished basis $\{p_1, \ldots, p_n\} = \mcb_{\tilde{P} (s)}$, the map $\tau_{i, i  + 2}$ is a cyclic $(k_{i + 1} - k_i)$ insertion. There is a unique total order $<$ on $F_{i, i + 2}$ such that the vanishing graph associated to $\{p_{k_i}, \ldots, p_{k_{i + 1}}\}$ is $ I_{\tau_{i, i + 2}, <}$. Furthermore, every cyclic $(k_{i + 1} - k_i)$ insertion $\sigma$ and total order $<$ arises as a monodromy map for some radar screen and regeneration of $\mathbf{w}_i$.
\end{thm}

The dictionary to use for the input structures of this theorem is as follows. The sets $F_{i , i + 1} (q), F_{i, i + 2}(q), F_{i + 1, i +2} (q^\prime)$ give $S_1, S_2, S_3$ respectively, the perturbation coefficients $\mathbf{c}$ gives a total order on $S_2$ and the radar screen gives $\sigma = \tau_{i, i + 2}$.

\begin{proof} We sketch a proof. Let $a = k_{i + 1} - k_i$, $b = k_i$ and recall that the LG-model $\tilde{\psi}_i (s, t)$ corresponding to $C_i$ converges to  $\mathbf{w}_i (u) = u^{a + b} + u^b + t$. As an $A^\prime$ sharpened pencil on $\C^*$, this is $[f(u):1] := [u^{a+b} + u^b : 1] = [-t : 1]$. We take a moment to understand the geometry of this elementary polynomial $f$. First observe that $\mathbf{w}_i$ has $a$ critical values at scaled roots of unity $d \zeta$ for $R = |b / (a + b)|^{1/a}$ and $\zeta \in \mu_a$, as well as a $(b-1)$ ramified critical value at $0$. Let $S^1_r$ be the radius $r$ circle and examine the contour $S_r := f^{-1} (S^1_r) \subset \C$ as we vary $r$. It is not hard to see that for $r > R$, $S_r$ is a circle which is an $(a + b)$-fold cover of $S^1_r$, while for $r < R$ it is a union of $a$ circles, $a$ of which cover $S^1_r$ once and the remaining circle covering it $b$ times. For $r = R$, $S_r$ is a circle with $a$ pinched pairs of points. This is illustrated in Figure \ref{fig:cont}.

\begin{figure}
\begin{picture}(0,0)%
\includegraphics{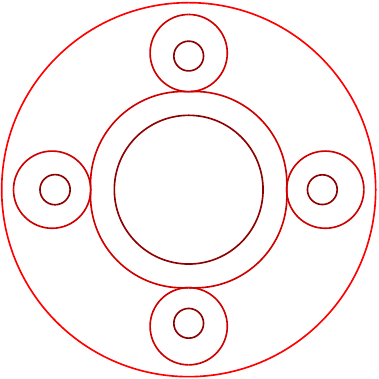}%
\end{picture}%
\setlength{\unitlength}{4144sp}%
\begingroup\makeatletter\ifx\SetFigFont\undefined%
\gdef\SetFigFont#1#2#3#4#5{%
  \reset@font\fontsize{#1}{#2pt}%
  \fontfamily{#3}\fontseries{#4}\fontshape{#5}%
  \selectfont}%
\fi\endgroup%
\begin{picture}(1726,1726)(758,-969)
\put(2001,-472){\makebox(0,0)[lb]{\smash{{\SetFigFont{5}{6.0}{\rmdefault}{\mddefault}{\updefault}{\color[rgb]{.82,0,0}$r=R$}%
}}}}
\put(997,409){\makebox(0,0)[lb]{\smash{{\SetFigFont{5}{6.0}{\rmdefault}{\mddefault}{\updefault}{\color[rgb]{1,0,0}$r>R$}%
}}}}
\put(1377, 56){\makebox(0,0)[lb]{\smash{{\SetFigFont{5}{6.0}{\rmdefault}{\mddefault}{\updefault}{\color[rgb]{.56,0,0}$r<R$}%
}}}}
\end{picture}%
\caption{\label{fig:cont} Contours for $|f (u)| = r$}
\end{figure}

Note that $\mathbf{w}_i$ is degenerate in the sense that it lies in the closure of the discriminant $\Delta_{A, A^\prime}$. Nevertheless, the sets $F_{i, i + 1} (q)$ and $F_{i, i + 2} (q)$ converge, up to a phase, to the $b$ roots of the inner circle of Figure \ref{fig:cont} and points contained in one of each of the $a$ outer circles, respectively. Were we to regenerate a straight line path from $q$ to $q^\prime$, it is not hard to see that the monodromy would then give a cyclic $b$-insertion $F_{i, i + 1} (q) \cup F_{i , i + 2} (q) \to F_{i + 1, i + 2} (q^\prime )$. 

To obtain the actual monodromy map $\tau_{i, i + 2}$, we need to define a radar screen $\mcb$ for the regeneration of $\mathbf{w}_i / \mu_{a}$ along $\tilde{\psi}_i (s, t)$. Recall from section \ref{sec:toric} that a radar screen is a distinguished basis of paths $ \mcb = \{p_1, \ldots, p_a, p_{a + 1}\}$ ending on the $(a + 1)$ critical values $\{q_1, \ldots, q_a, 0\}$ of $\tilde{\psi}_i (s, \_ )$, ordered so that $|q_i| > |q_{i + 1}|$. It is determined uniquely by the regeneration $\tilde{\psi}_i$ and a choice of logarithmic lifts $\{\tilde{q}_1, \ldots, \tilde{q}_a\}$ of the critical values.  We note that, to first order, only the $(c_{a + b - 1 }, \ldots, c_{b + 1} )$ projection of  the coefficient $\mathbf{c} = (c_{n + 1}, \ldots, c_0)$, matters in determining the norm ordering of the critical values for $\tilde{\psi}_i (s, \_ )$. It is easy to see that one may prescribe any ordering with a judicious choice of such coefficients.

The key point in the proof is that for $s \ll 1$, the Figure \ref{fig:cont} is only mildly modified, so that instead of pinching off $a$ circles at once when $r = R$, we pinch off circles one by one, each time $r = |q_i|$. Through the identification of $F_{i, i + 2}$ with the outer contours given above, we see that the ordering of critical values gives a total ordering $<$ of $F_{i, i + 2}$. Between each pinch, the choice of logarithmic branch $\tilde{q}_i = \log (q_i)$ for the radar screen $\mcb$ has the effect of rotating the circle $S_r$ as one performs monodromy along $p_{i + 1}$. Note that this monodromy preserves the cyclic ordering of the $b$ points that survive the pinching, so that the total monodromy of $F_{i, i + 1} (q)$ along $\delta_{a + 1}^{-1}$ also preserves the cyclic order. These observations show that $\tau_{i, i + 2}$ is a cyclic $a$ insertion and $F_{i, i + 2}(q)$ is totally ordered.

To establish the claim about the vanishing tree, simply observe that, for each pinch, we add a vanishing cycle connecting two points of the fiber $f^{-1} (p_j (z))$ on the central component of $S_r$. One of the points will always be the point pinched off, and the other will be one of its cyclic neighbors. The fact that this is always the clockwise neighbor corresponds to our choice of counter-clockwise orientation for the radar screen distinguished basis. It is left as an exercise to see that the resulting collection of pairs of points is $I_{\tau_{i, i + 2}, <}$.
\end{proof}

Applying this to the example $J = \{0,1,2,4,8\}$ illustrated in Figure \ref{fig:degeneration}  with fundamental radar screen gives the vanishing tree in Figure \ref{fig:vtr}. One starts with the unique $(1,1)$ cyclic insertion yielding the green vanishing thimble, proceeds to a $(2, 2)$ insertion which gives two red vanishing thimbles, and completes the tree with a $(4, 4)$ insertion producing the $4$ blue vanishing thimbles. In general, the proof above shows that the insertions can be chosen arbitrarily. However, if we choose the fundamental radar screen distinguished basis for an exponential sequence $J = \{0\} \cup \{2^i : 0\leq i \leq m\}$, it can be shown that each insertion is a perfect shuffle \cite{perfectshuffle}. This reflects a general phenomenon that the fundamental radar screen gives insertions $\tau_{i, i + 1}$ that maximally separate the points in $F_{i, i+ 1} (q)$ and $F_{i, i + 2} (q)$.

\begin{figure}[h]
\begin{picture}(0,0)%
\includegraphics{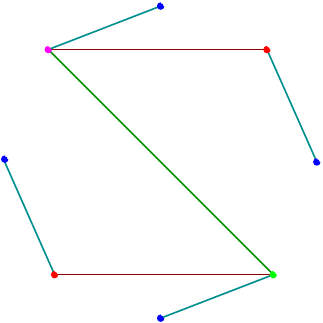}%
\end{picture}%
\setlength{\unitlength}{4144sp}%
\begingroup\makeatletter\ifx\SetFigFont\undefined%
\gdef\SetFigFont#1#2#3#4#5{%
  \reset@font\fontsize{#1}{#2pt}%
  \fontfamily{#3}\fontseries{#4}\fontshape{#5}%
  \selectfont}%
\fi\endgroup%
\begin{picture}(1469,1469)(1359,-1313)
\end{picture}%
\caption{\label{fig:vtr} A vanishing tree for $J = \{0, 1, 2, 4, 8\}$.}
\end{figure}

\subsection{Interpretations of $A_n$ degenerations}

We conclude this section with some observations and corollaries of Theorem \ref{thm:an}. First note that, were we to find the actual exceptional collection associated to the vanishing tree, we would need to include gradings on each of the edges. For ease of exposition, we neglect these gradings, commenting only that choosing different logarithmic branches in a radar screen that yield equivalent vanishing trees will generally alter the graded version. 

Now, recall that Gabriel's theorem classifies quivers of finite type as directed Dynkin diagrams \cite{gabriel}. The set of quivers $\mcq$ whose the underlying graph is $A_n$, can be identified with the power set of $ \{ 2, \cdots, n\}$. To obtain a precise correspondence, order the vertices of the quiver by $\{v_1, \ldots, v_n\}$ and edges $\{e_2, \ldots, e_n\}$ where $e_i = \{v_{i - 1}, v_i\}$. We say $\mathfrak{o} (e_i ) = \pm 1$ if $e_i$ is directed towards $i$ or  $(i -1 )$ respectively. Then map $J = \{k_1, \ldots, k_{m - 1}\} \subset \{2, \ldots, n\}$ to the unique quiver $\Gamma_J$ which satisfies $\mathfrak{o} (e_{i}) = -1$ if and only if $i \in J$.

Given such a quiver, we propose that there is a natural $T$-structure on the derived category $\mcd$ of right modules over the path algebra $\mcp (\Gamma_J )$ (recall that, up to equivalence, $\mcd$ is independent of $J$). First, write $P_i$ for the projective module of all paths with target $v_i$. Let $p_J : \{1, \cdots, n \} \to \Z$ be the function $p_J (j) = \frac{j }{2} - \frac{1}{2} \sum_{i = 1}^j \mathfrak{o} (e_{i + 1})$. We view $p_J$ as a perversity function and define $(\mcd^{\geq 0}_J , \mcd^{\leq 0}_J )$ as subcategories for which a bounded chain complex $C^*$ of right $\mcp (\Gamma_J)$ modules is in $\mcd^{\geq 0}_J$ if $H^k (\Hom (P_i , C^*)) = 0$ for all $k < p_J (i)$ and likewise for $\mcd^{\leq 0}_J$.

Let us now construct a complete exceptional collection $\mathbf{E}_J = \langle E_1, \ldots, E_n \rangle$ in the heart $\mcd^{\geq 0}_J \cap \mcd^{\leq 0}_J$ for a given $J$.  Define $E_i = P_i[p_J (i)]$ if $\mathfrak{o} (e_{i + 1}) = 1$ and  
\begin{equation*} E_i = (0 \leftarrow P_i[p_J(i)] \stackrel{e_{i + 1}}{\longleftarrow} P_{i + 1}[p_J(i + 1)] \leftarrow 0) \end{equation*}
otherwise.

Given an exceptional collection $\mathbf{E} = \langle E_1, \ldots, E_n \rangle $, write $R_{\mathbf{E}} = \ext^* (\oplus_1^n E_i , \oplus_1^n E_i )$ for its Yoneda algebra. A simple computation gives the following proposition.

\begin{prop} The collection $\mathbf{E}_J$ is a complete, strong exceptional collection for $\mcd$. The heart of $(\mcd^{\geq 0}_J, \mcd^{\leq 0}_J)$ is equivalent to the category of finitely generated right modules over $R_{\mathbf{E}_J}$. \end{prop}

To connect this collection to maximal degenerations of Landau-Ginzburg models, we approach the fixed point  $\psi_J \in \mlg{A}{A^\prime}$ associated to $J$ via a degeneration path $\psi (\mathbf{c}, s, \_)$. Using Theorem \ref{thm:an}, we can describe the vanishing tree of a radar screen $\mcb$ through a sequence of totally ordered, cyclic $(k_{i + 1} - k_i)$-insertions. Partitioning $\{1, \cdots , n + 1\}$ to the ordered sets $S_i = \{k_{i - 1} + 1, \ldots, k_{i}\}$ for $1 \leq i \leq m$, we identify 
\begin{equation*}
F_{i, i + 1}  = \{1, k_1, k_1 - 1, \ldots, k_0 + 1, k_2,  \ldots, k_1 + 1, \ldots \ldots, k_{i - 1} , \ldots , k_{i - 2} + 1 \}
\end{equation*}
where the cyclic order is as written. We define the cyclic insertions  $\sigma_i : F_{i, i + 1}  \cup S_{i + 1} \to F_{i + 1, i + 2}$ as the inclusion. Let $R(J)$ be the radar screen which yields this data and write the resulting exceptional collection as $\mathbf{E}_{R (J)}$.

\begin{prop} For every $J \subset \{2, \cdots , n\}$,  $R_{\mathbf{E}_{R (J)}} \approx R_{\mathbf{E}_J}$.
\end{prop}

This proposition suggests the space $\mlg{A}{A^\prime}$ is connected to the space of stability conditions for $\mcd$. In the $A_n$ case, near maximal degeneration points, we obtain $T$-structures for the triangulated category of the LG-model which relate directly to abelian categories of directed $A_n$ quivers. In other words, we have categorified the bijection between fixed points of the monotone path stack and directed $A_n$ quivers to an equivalence between an exceptional collection of a degeneration near the fixed point and an exceptional collection naturally associated to the directed quiver. Moving from one fixed point to another along certain edges of the monotone path polytope crosses a wall in the norm stratification which results in Coxeter functors, or tiltings, of the ambient triangulated category. 

We end the section on the $A_n$ case by a brief comment on homological mirror symmetry. It is known that the homological mirror category for $A_n$ is the graded derived category of singularities for $z^n : \C \to \C$ \cite{HV}. One can view this category as a weighted divisor blow-up of the origin in $\C$ with weight $n$. The monotone path associated to $J$ may then be viewed as mirror to a sequence of $m$ blow ups with weights $(k_{i + 1} - k_i)$. This perspective fits well with birational mirror symmetry landscape discussed in \cite{DKK} and \cite{katzarkov}.

\section{Three point blow up of $\p^2$}

We conclude this paper with an example of a different flavor than previous sections. Throughout, let $X_3$ denote a smooth del Pezzo surface of degree 6; that is, a blow up of $\mathbb{P}^2$ at three distinct non-collinear points. This space is mirror (and isomorphic) to $X_{\examplep{2}}$ given in section \ref{sec:toric}. The case of degree 7 was considered in \cite[Section 5]{DKK}.  Recall that $\textrm{Pic}(X_3)\otimes\mathbb{R}$ is spanned by the pull-back of the hyperplane class and the exceptional divisors $E_1$, $E_2$, $E_3$ corresponding to the blown-up points, and that the effective cone $\textrm{Eff}(X_3)$ is generated by $E_1$, $E_2$, $E_3$, along with the pull-backs of the lines through the pairs of points, $E_{12}$, $E_{13}$, $E_{23}$. The effective cone admits a chamber decomposition into Zariski chambers, with each maximal chamber corresponding to a birational model obtained from $X_3$ by birational contractions; moreover, the codimension $1$ external walls of $\textrm{Eff}(X_3)$, equipped with this decomposition, correspond to Mori fibrations obtained from $X_3$, and the codimension 2 external walls correspond to Sarkisov links between the fibrations. We refer to \cite{hacon} for a general discussion of Mori fibrations and Sarkisov links from the perspective of chambers. 

The structure of the external walls of $\textrm{Eff}(X_3)$ was considered in particular by Kaloghiros in \cite[Example 4.7]{kalo},  as a special case of a substantially more general result concerning codimension 3 external walls and relations amongst Sarkisov links. It is convenient to consider a dual graph $\Gamma_3$, with vertices corresponding to the codimension 1 external walls (i.e. the Mori fibrations) and edges corresponding to the codimension 2 external walls (i.e. Sarkisov links). A picture of this graph appears in \cite[Figure 6]{kalo}. By inspection this graph is observed to be the edge graph of the 3-dimensional associahedron. This is consistent with toric mirror symmetry and the results of \cite[Section 5]{DKK}, as the associahedron appears as a facet of the secondary polytope of the point configuration $\example{2} \subset \mathbb{Z}^2$ which is the Batyrev mirror of $X_3$.

\begin{figure}
\includegraphics[scale=.23]{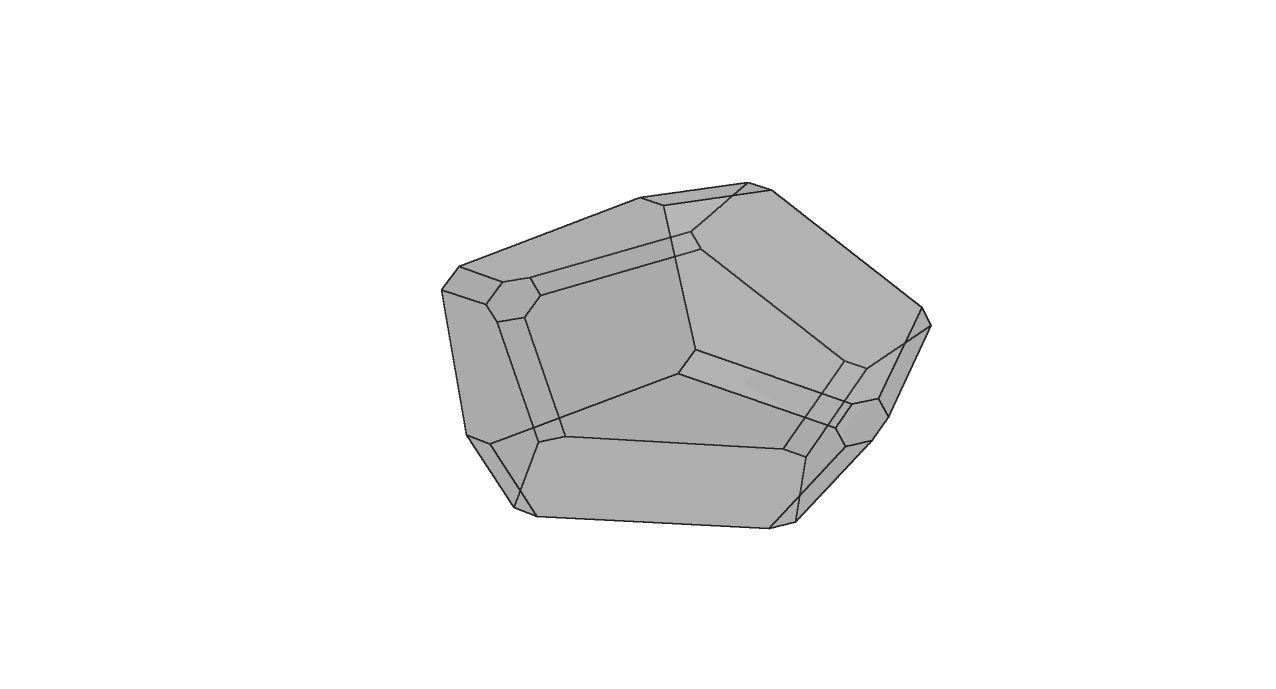}
\caption{\label{fig:mpp}The monotone path polytope for $(\examplep{2}, \example{2})$}
\end{figure}

As noted in \cite{kalo} the graph $\Gamma_3$ has 14 vertices, which correspond to Mori fibrations as follows:
\begin{itemize}
\item[(i)] 2 vertices correspond to to the trivial fibration $\mathbb{P}^2\rightarrow\{\textrm{pt}\}$, where $\mathbb{P}^2$ is obtained from $X_3$ by blowing down $E_1,E_2,E_3$, respectively $E_{12},E_{13},E_{2,3}$. 
\item[(ii)] 6 vertices correspond each to the fibration $\mathbb{F}_1\rightarrow\mathbb{P}^1$, where the map $X_3\rightarrow \mathbb{F}_1$ factors through the blow down of one of $E_1,E_2,E_3,E_{12},E_{13},E_{23}$. 
\item[(iii)] 6 vertices correspond each to the fibration $\mathbb{P}^1\times\mathbb{P}^1\rightarrow\mathbb{P}^1$, where the map $X_3\rightarrow\mathbb{P}^1\times\mathbb{P}^1$  factors through a blow down of one of $E_1,E_2,E_3$, and one of the two projections $\mathbb{P}^1\times\mathbb{P}^1\rightarrow\mathbb{P}^1$ is fixed. 

\end{itemize}
On the other hand, the monotone path polytope of the secondary polytope of $A$ with respect to the $\{0\}$-sharpening is of small enough complexity to be constructed via software. A picture of the resulting truncated associahedron appears in Figure \ref{fig:mpp}.  We observe that it has 36 vertices. Qualitatively, they correspond to the possible choices in the above description of $\Gamma_3$. 
\begin{itemize}
\item[(i)] 12 vertices correspond to the trivial fibration $\mathbb{P}^2\rightarrow\{\textrm{pt}\}$, where $X_3\rightarrow\mathbb{P}^2$ is one of the six ordered blow-downs of $E_1,E_2,E_3$, or one of the six ordered blow-downs of $E_{12},E_{13},E_{23}$. 
\item[(ii)] 12 vertices correspond to the fibration $\mathbb{F}_1\rightarrow\mathbb{P}^1$, where the map $X_3\rightarrow\mathbb{F}_1$ is given by an ordered blow-down of two of $E_1,E_2,E_3$, or an ordered blow-down of two of $E_{12},E_{13},E_{23}$.
\item[(iii)] 12 vertices correspond to the fibration $\mathbb{P}^1\times\mathbb{P}^1\rightarrow\mathbb{P}^1$, where the map $X_3\rightarrow\mathbb{P}^1\times\mathbb{P}^1$  factors through a blow down of one of $E_1,E_2,E_3$, and a blow down of one of $E_{12},E_{13},E_{23}$ not disjoint to $E_i$, and one of the two projections $\mathbb{P}^1\times\mathbb{P}^1\rightarrow\mathbb{P}^1$ is fixed. \end{itemize}
We observe that a vertex of the dual graph representing a Mori fibration is replaced with the collection of full runs of the minimal model program on $X_3$ whose last birational map is that Mori fibration. As was conjectured in \cite{DKK}, the semi-orthogonal decompositions of $D^b (X_3)$ arising from such runs are then conjectured to yield equivalent subcategories to those arising from the maximal degenerations of the mirror LG model.

We conclude with a brief discussion of prospects for extending beyond the toric case, and in particular to del Pezzo surfaces $X_k$ of degrees 1 through 6. The birational geometry of these surfaces is classical, though intricately structured \cite{manin}.  Motivated by the Hori-Vafa ansatz, \cite{AKO} posited the mirror in each case to be a Landau-Ginzburg model $f_k : Y_k \rightarrow \mathbb{P}^1$ of a rational elliptic surface $Y_k$ with prescribed fiber at $\infty$. They verified homological mirror symmetry in the form $D^b{X_k}\cong \fs{f_k}$; however, the identification of the K\"{a}hler moduli of $X_k$ with the complex moduli of $Y_k$ was not pursued.  This identification was completed in unpublished work of Pantev \cite{pantev}. 

In general, if $f: Y \to \mathbb{P}^1$ is a compactified LG model, results from \cite{KKSP} show that the complex  $T_{Y,Y_\infty} \to   f^*(T_{\mathbb{P}^1,\infty})$ defining perturbations of $f$ that fix the fiber at infinity can be integrated to produce a smooth moduli stack $\mcm$ of LG models. It is not hard to see that, when $f$ arises as a sharpened pencil, the quotient of $\mcm$ by the action of $\C^* \times \C$ naturally embeds as a substack of $\mlg{A}{A^\prime}$. In the toric cases, $Y_k$ can be obtained from the Batyrev mirror family by explicit blow-ups and we have seen that $\mlg{A}{A^\prime}$ is a natural geometric compactification of the complex moduli of the Batyrev mirror. While we suspect a similar nested compactification exists in the non-toric cases, they do not appear to have been studied from this vantage point. However, see \cite{looijenga} for a thorough study of compact moduli of rational elliptic surfaces. The recent investigations of Donaldson \cite{donaldson} regarding $K-$ and $b-$ stability of Fano manifolds will be relevant, replacing the role that classical geometric invariant theory plays in constructing the chamber decomposition on the effective cone. 

The above considerations provide a convenient way of studying surfaces whose derived category is close to being generated by an exceptional collection. In particular our analysis suggests the following conjecture.

\begin{conj} The derived category of the Barlow surface is not generated by an exceptional collection.
\end{conj}

The proof of this conjecture will lead to examples of nontrivial categories with trivial K - theory.

\bibliography{thebib}
\bibliographystyle{plain}
\end{document}